\documentclass[11pt,a4paper]{article}

\usepackage{amsmath}
\usepackage{mathtools}
\usepackage{amssymb}
\usepackage{amsthm}
\usepackage{mathrsfs}
\usepackage{tikz-cd}
\usepackage{enumitem}
\usepackage[bb=boondox]{mathalfa}
\usepackage[margin=2cm]{geometry}

\newcommand{\abstrcompldom}{\mathfrak{d}^{c}}
\newcommand\initial{\mathbb{0}}
\newcommand\Fam{\mathbf{Fam}}
\newcommand\Bool{\mathbf{Bool}}

\newcommand\Famm[1]{\mathbf{Fam}\left(#1\right)}

\newcommand\Conn{\mathbf{Conn}}
\newcommand\fin{\mathsf{fin}}
\newcommand\pb{\mathsf{pb}}
\newcommand\FinFam{\mathbf{FinFam}}
\newcommand\Cat{\mathbf{Cat}}
\newcommand\CAT{\mathbf{CAT}}
\newcommand\Set{\mathbf{Set}}
\newcommand\finSet{\mathbf{finSet}}
\newcommand\Topl{\mathbf{Top}}
\newcommand\LocConTopl{\mathbf{LocConTop}}

\newcommand\PsAlg{\mathbf{PsAlg}}
\newcommand\Dist{\mathbf{Dist}}
\newcommand\wCPO{\omega\dash\mathbf{CPO}}

\newcommand{\cat}[1]{\mathbb{#1}}
\newcommand{\terminal}{\mathbb{1}}
\let\lim\relax
\DeclareMathOperator*\lim{\mathsf{lim}}
\DeclareMathOperator*\colim{\mathsf{colim}}

\newcommand{\Tt}{\mathcal T}

\newcommand{\TPsAlg}{\Tt\dash\PsAlg}
\newcommand{\FamT}{\Fam_\Tt}
\newcommand{\yy}{\mathfrak y}
\newcommand{\uu}{\mathfrak u}
\newcommand{\uuo}{\overline{\uu}}

\newcommand{\Ll}{\mathcal L} 
\newcommand{\Pp}{\mathcal P}
\newcommand{\Lfin}{\Ll_\fin}
\newcommand{\Lpb}{\Ll_\pb}

\newcommand{\Colim}{\mathbf{Colim}}
\newcommand{\Lim}{\mathbf{Lim}}
\newcommand{\PhiLim}{\Phi\dash \Lim}
\newcommand{\PhiColim}{\Phi\dash \Colim}

\renewcommand{\emptyset}{\varnothing}
\newcommand{\iso}{\cong}
\newcommand{\eqv}{\simeq}
\newcommand{\adj}{\dashv}
\newcommand{\dash}{\text{-}}
\newcommand{\comma}{\downarrow}
\newcommand{\pt}{\ast}
\newcommand{\op}{\mathsf{op}}
\DeclareMathOperator{\ob}{\mathsf{ob}}
\newcommand{\expn}{\Rightarrow}

\theoremstyle{plain}
\newtheorem{lemma}{Lemma}[section]

\newtheorem{corollary}[lemma]{Corollary}
\newtheorem{theorem}[lemma]{Theorem}

\theoremstyle{definition}
\newtheorem{definition}[lemma]{Definition}
\newtheorem{remark}[lemma]{Remark}

\let\sum\relax
\newcommand{\sum}{\coprod}

\title{Free extensivity via distributivity}
\author{Fernando Lucatelli Nunes, Rui Prezado, Matthijs Vákár}
\date{}

\begin{document}

\maketitle

\abstract{We consider the canonical pseudodistributive law between various
free limit completion pseudomonads and the free coproduct completion
pseudomonad. When the class of limits includes pullbacks, we show that this
consideration leads to notions of extensive categories. More precisely, we
show that \textit{extensive categories with pullbacks} and \textit{infinitary
lextensive categories} are the pseudoalgebras for the pseudomonads resulting
from two of these pseudodistributive laws. Moreover, we introduce the notion
of \textit{doubly-infinitary lextensive category}, and we establish that the
freely generated ones are cartesian closed. From this result, we further
deduce that, in freely generated infinitary lextensive categories, the objects
with a finite number of connected components are exponentiable. We conclude
our work with remarks on examples, descent theoretical aspects of this work,
results concerning non-canonical isomorphisms, and relationship with other
work. 
 }

\section*{Introduction}

Two-dimensional monad theory~\cite{BKP89, Lac00, Bou14, Luc16} is the
categorical approach to bidimensional universal algebra, which mainly deals with
the problem of understanding \textit{algebraic structures}, in a suitable sense,
over objects in a $2$-category.

Focusing on the case where the base $2$-category is the $2$-category of
categories $\CAT$, this leads to the systematic study of categories with
additional \textit{(algebraic)} structures (or properties)~\cite{BKP89, KL97, Luc16}.
The $2$-categories of interest usually arise as $2$-categories of pseudoalgebras
or lax algebras of a given pseudomonad -- we refer, for instance,
to~\cite{Lac02, Luc18b} for the definitions of these concepts.

There are many well-known examples of such $2$-categories of interest, namely:
\begin{itemize}[label=--]
    \item 
        the 2-category of monoidal categories, monoidal functors and monoidal
        natural transformations is the $2$-category of pseudoalgebras for the
        free monoid 2-monad (also known as the \textit{list} 2-monad)  on $\CAT
        $, \textit{e.g.} \cite{BKP89, Her00, Luc19};
    \item 
        the 2-category of monads is given by the 2-category of lax algebras
        w.r.t. the identity $2$-monad on $\CAT$, \textit{e.g.}
        \cite[pag.~33]{Luc18x} and \cite{Luc18b};
    \item 
        $2$-categories of pseudofunctors and pseudonatural transformations
        between two suitable $2$-categories with weighted bicolimits is given
        by the $2$-category of pseudoalgebras w.r.t. a pseudomonad induced by
        a suitable pseudo-Kan extension, \textit{e.g.} \cite{Luc16, Luc18a}; 
    \item  
        the $2$-category of categories with \( \Phi \)-(co)limits and \(
        \Phi\)-(co)limit preserving functors is the 2-category of
        pseudoalgebras and pseudomorphisms w.r.t. a suitable free (co)limit
        completion pseudomonad on $\CAT $, \textit{e.g.} \cite{KL97, Mar97,
        PCW00, Luc19}.
\end{itemize}
 
The framework of two-dimensional monad theory is well-suited for studying the
age-old problem of distributivity between limits and colimits of a given
category. Specifically, our focus lies on the canonical
\textit{pseudodistributive law}~\cite{Mar99, Mar04} between various sorts of
\textit{free limit completion} pseudomonads and the \textit{free coproduct
completion} pseudomonad. Previous considerations of such distributivity
properties include (infinitary) distributive categories~\cite{CLW93},
completely distributive categories~\cite{MRW12}, and doubly-infinitary
distributive categories~\cite{LV24}. In this paper, we show that a similar
analysis gives rise to well-known and novel notions of \textit{extensive
categories}. 

Recall that, if \( \cat C \) has (in)finite coproducts, \( \cat C \) is said
to be an (\textit{infinitary}) \textit{extensive category}~\cite{CLW93} if the
canonical functor 
\begin{equation*}
    \begin{tikzcd}
        \displaystyle\prod\limits_{i \in I} \cat C \comma X_i 
            \ar[r,"\sum"]
        & \cat C \comma \displaystyle\sum\limits_{i \in I} X_i
    \end{tikzcd}
\end{equation*}
is an equivalence of categories for every (in)finite family \( (X_i)_{i \in I}
\) of objects in \( \cat C \). It has been observed in~\cite{CV04}
and~\cite[Section 7]{PL23b} that ``(infinitary) extensivity'' can be viewed as
a distributivity condition of pullbacks over (infinitary) coproducts.

The present work, which is a sequel to~\cite{LV24}, aims to study categories
with a given class of limits, small coproducts, and a (pseudo)distributive law
between them. More precisely, given a class \( \Phi \) of diagrams, we remark
that there is a canonical pseudodistributive law between the free \( \Phi
\)-limit completion pseudomonad and the free coproduct completion, denoted by
\( \Fam \)~\cite{Zob76, Koc95, Mar97}. We show that the pseudoalgebras for the
composite pseudomonad can be easily described; namely, they can be given as
the categories with \( \Phi \)-limits and coproducts such that the coproduct
functor
\begin{equation}
    \label{eq:intro.coprod}
    \sum \colon \Fam(\cat C) \to \cat C
\end{equation} 
preserves \( \Phi \)-limits. 

Our key contribution is the observation that various flavors of infinitary
extensive categories are pseudoalgebras for such composites of pseudomonads.
More precisely, assuming that $\cat C$ is a category with coproducts:
\begin{itemize}[label=--]
    \item
        if \( \cat C \)  has pullbacks, and~\eqref{eq:intro.coprod} preserves
        them, then \( \cat C \) is \textit{infinitary extensive with pullbacks};
    \item
        if \( \cat C \) has finite limits, and~\eqref{eq:intro.coprod} preserves
        them, then \( \cat C \) is \textit{infinitary lextensive};
    \item
        if \( \cat C \)  has small limits, and~\eqref{eq:intro.coprod} preserves
        them, we say that the category $\cat C$ is \textit{doubly-infinitary
        lextensive}. We observe that $\cat C $ satisfies such properties if and
        only if $\cat C $ is simultaneously a \textit{doubly-infinitary
        distributive} category~\cite{LV24} as well as a \textit{lextensive}
        category. 
\end{itemize}

The observations presented above, coupled with the findings of \cite{LV23},
contribute to the understanding of extensive categories and distributive
categories through the prism of $2$-dimensional universal algebra, adding to
the comparison of these notions 
originally started in \cite{CLW93}. 

In \cite{LV24}, it was demonstrated that freely generated doubly-infinitary
distributive categories are cartesian closed. Furthermore, this investigation
extended to encompass the study of exponentials in freely generated infinitary
distributive categories. More generally, in \cite{LV24b}, a comprehensive
analysis was conducted, yielding general results concerning exponentiability
and cartesian closedness of Grothendieck constructions. Notably, these results
are applicable to a wide array of contexts, including freely generated
categorical structures.

Motivated by~\cite{LV23,LV24b}, we further study the exponentiable objects of
the free pseudoalgebras for the pseudomonads we considered; namely, we find
that:
\begin{itemize}[label=--]
  \item 
    in a freely generated infinitary lextensive category, objects with a
    finite number of connected components are exponentiable;
  \item 
    freely generated doubly-infinitary extensive categories  are
    \textit{cartesian closed}.
\end{itemize}

\paragraph*{Outline:} We revisit the notion of free \( \Phi \)-colimit
completions for a class \( \Phi \) of diagrams (small categories) in
Section~\ref{sect:free.comp}. Several authors have worked on free (co)limit
completions; namely, we have~\cite{GV72, Tho84, AR20} for ordinary categories,
and~\cite{Kel82, AK88} in the context of enriched category theory. We also
have the accounts~\cite{Koc95, Zob76, Mar97, PCW00} which study free \( \Phi
\)-(co)limit completions from the perspective of 2-dimensional monad
theory~\cite{BKP89, Luc16, Luc18b, Luc19}, which is the approach we employ, so
some familiarity with these methods is assumed. We focus specifically on four
classes of free (co)limit completions:

\begin{itemize}[label=--]
    \item
        the free \textit{coproduct} completion, denoted \( \Fam \),
    \item 
        the free \textit{finite limit} completion, denoted \( \Lfin \), 
    \item
        the free \textit{pullback} completion, denoted \( \Lpb \),
    \item 
        the free \textit{small limit} completion, denoted \( \Ll \).
\end{itemize}

In Section~\ref{sect:pseudomonad}, we study the distributivity of \( \Phi
\)-limits over coproducts. Similar work has been carried out in~\cite{ARV01,
MRW12, Gle18} and in the prequel~\cite{LV24}. After recalling the necessary
concepts pertaining to pseudodistributive laws~\cite{Mar99, Mar04, Wal19}, we
confirm that there is a pseudodistributive law between any free \( \Phi \)-limit
completion pseudomonad and the free coproduct completion pseudomonad \( \Fam \)
(Lemma~\ref{lem:fam.lifts}). Instantiating this result with each of the
aforementioned free limit completions, we obtain the composite pseudomonads \(
\Fam \circ \Lfin \), \( \Fam \circ \Lpb \), and \( \Fam \circ \Ll \).

The study of these pseudomonads and their pseudoalgebras have given us novel
characterizations of (infinitary) extensivity. More specifically, we prove that:
\begin{itemize}[label=--]
    \item
        \( (\Fam \circ \Lfin) \)-pseudoalgebras are precisely the lextensive
        categories (Theorem~\ref{thm:inf.lext}),
    \item
        \( (\Fam \circ \Lpb) \)-pseudoalgebras are precisely the extensive
        categories with pullbacks (Theorem~\ref{thm:inf.ext.pb}).
\end{itemize}

Moreover, in Section \ref{subsect:free.dbl.inf.lext}, we introduce the notion
of  \textit{doubly-infinitary lextensive categories}: these are the  \( (\Fam
\circ \Ll) \)-pseudoalgebras. Finally, we prove in
Theorem~\ref{thm:db.inf.lex} that doubly-infinitary lextensive categories
correspond to lextensive categories that are also doubly-infinitary
distributive as introduced in \cite{LV24}.

Mainly motivated by \cite{LV24, LV24b}, in the present work, our study
exponentiable objects in freely generated categorical structures is
the content of Section~\ref{sect:expn}. This includes our main results,
which respectively state that:
\begin{itemize}[label=--]
    \item
        freely generated doubly-infinitary lextensive categories are
        \textit{cartesian closed} (Theorem~\ref{thm:main.two}),
    \item 
        in freely generated infinitary lextensive categories, \textit{finite
        coproducts of connected objects} are exponentiable
        (Theorem~\ref{thm:main.one}). 
\end{itemize}

In Section \ref{sec:Examples}, we discuss examples of (doubly)-infinitary
lextensive categories.  Finally, in Section~\ref{sect:epilogue}, we show that
analogous results also hold for the \textit{free finite coproduct completion}
pseudomonad, leading to similar characterisations of (finitely) extensive
categories. Further, we discuss possible avenues for future work, descent
theoretical considerations of our findings, and we note a result on
non-canonical isomorphisms, as a direct consequence of the work
of~\cite{Luc19}.

\paragraph*{Acknowledgements} We acknowledge the community for the prompt
reactions to our work. In particular, we thank Robin Kaarsgaard for useful
queries, which brought forth the final topic of Section \ref{sect:epilogue}.
We are grateful to the anonymous referees for their informative reports.

The first author wishes to extend a special thanks to Tim Van der Linden for
his warm hospitality at Université catholique de Louvain during a brief stay
in May 2024, for the UCLouvain-ULB-VUB Category Theory Seminar.  The inspiring
environment and their kindness nurtured a much needed peace of mind, which
aided the advancement of this work.

\paragraph*{Funding.} This project has received funding via NWO Veni grant
number VI.Veni.201.124.

The first two named authors acknowledge partial financial support by 
\textit{Centro de Matemática da Universidade de Coimbra} (CMUC), funded by the
Portuguese Government through FCT/MCTES, DOI 10.54499/UIDB/00324/2020.

\section{Free colimit completions}
\label{sect:free.comp}

Let \( \CAT \) be the 2-category of locally small (\(\Set\)-enriched)
categories. Any other category considered in this work is assumed to be an
object of \( \CAT \). 

Let \( \Phi \) be a class of small categories.  We say that a category \( \cat C \)
has \( \Phi \)-colimits if any functor \( D \colon \cat J \to \cat C \) with
\( \cat J \in \Phi \) has a colimit in \( \cat C \).  Moreover, if \( F \colon
\cat C \to \cat D \) is a functor between categories with \( \Phi \)-colimits, we
have a morphism
\begin{equation}
    \label{eq:colim.comp}
    \colim FD \to F(\colim D) 
\end{equation}
which is natural in \( D \colon \cat J \to \cat C\) for \( \cat J \in \Phi \). We
say that \( F \) preserves \( \Phi \)-colimits if \eqref{eq:colim.comp} is a
natural isomorphism.

We let \( \PhiColim \) be the 2-category of categories with \( \Phi \)-colimits,
\( \Phi \)-colimit preserving functors and natural transformations. We have
a forgetful 2-functor
\begin{equation}
    \label{eq:phicol.forget}
    \begin{tikzcd}
        \PhiColim \ar[r] & \CAT 
    \end{tikzcd}
\end{equation}
which is pseudomonadic -- we let \( \Pp_\Phi \) be the left biadjoint to
\eqref{eq:phicol.forget}, as well as the induced pseudomonad by the biadjunction
-- the \textit{free \( \Phi \)-colimit completion} pseudomonad. We can justify
this abuse of notation, by noting that a category \( \cat C \) has \( \Phi
\)-colimits if and only if the (fully faithful) unit of \( \Pp_\Phi \) at \(
\cat C \), denoted by \( \yy \colon \cat C \to \Pp_\Phi(\cat C) \), has a left
adjoint \cite{AK88}. Thus, being a \( \Pp_\Phi \)-pseudoalgebra is a
\textit{property} of the category \( \cat C \), as opposed to structure
\cite{KL97}. In other words, \( \Pp_\Phi \) is a \textit{lax idempotent
pseudomonad} \cite{Koc95,Mar97,PCW00, LUCCLE, Lurdes2024} (also known as
\textit{Kock-Zöberlein pseudomonad}), and, hence, a property-like
pseudomonad~\cite{KL97, Luc18a}. 

Dually, we say that a category \( \cat C \) has \( \Phi \)-limits whenever \(
\cat C^\op \) has \( \Phi \)-colimits, and we say that a functor \( F \colon \cat C \to
\cat D \) between categories with \( \Phi \)-limits preserves \( \Phi \)-limits
if \( F^\op \colon \cat C^\op \to \cat D^\op\) preserves \( \Phi \)-colimits. We
denote by \( \PhiLim \) the 2-category of categories with \( \Phi \)-limits, \(
\Phi \)-limit preserving functors and natural transformations.
We also have a
pseudomonadic 2-functor
\begin{equation}
    \label{eq:philim.forget}
    \begin{tikzcd}
        \PhiLim \ar[r] & \CAT
    \end{tikzcd}
\end{equation}
whose left biadjoint and induced pseudomonad are denoted by \( \Ll_\Phi \), so
that we have a biequivalence \( \Ll_\Phi \dash \PsAlg \eqv \PhiLim \). In fact,
we note that \(\Ll_\Phi(\cat C) = \Pp_\Phi(\cat C^\op)^\op \). We likewise
denote the (fully faithful) unit at a category \( \cat C \) by \( \yy \colon
\cat C \to \Ll_\Phi(\cat C) \). This unit has a right adjoint if and only if
\(\cat C\) has $\Phi$-limits.

\begin{remark}
    In \cite{Kel82, AK88}, the notions of \( \Phi \)-colimits and \( \Phi
    \)-colimit completions were worked out in the more general setting of
    enriched category theory, where \( \Phi \) is taken to be a class of small
    weights instead (that is, functors \( \cat J ^\op \to \cat V \) with \( \cat
    J \)   small), where \( \cat V \) is the base monoidal category. 

    In our setting, the notions we provided correspond to the classes \( \Phi \)
    of weights that are constant functors to the terminal object. We leave the
    consideration of our results in an enriched setting for future work. 
\end{remark}

As argued in \cite{Kel82, AK88}, the free \( \Phi \)-colimit completion
\(\Pp_\Phi(\cat C)\) of a category \( \cat C \) is most succinctly described as
the smallest full subcategory of \( \CAT(\cat C^\op, \Set) \) that has \( \Phi
\)-colimits. Dually, \(\Ll_\Phi(\cat C)\) is the smallest full subcategory of \(
\CAT(\cat C, \Set)^\op \) that has \( \Phi \)-limits. With this, we can obtain
an expression for the hom-sets of \( \Phi \)-(co)limit completions:

\begin{lemma}
    \label{lem:phi.irreducible}
    Let \( \Phi \) be a class of small categories, let \( C \) be an object
    of \(\cat C \), and let \( E \colon \cat K \to \Pp_\Phi(\cat C) \) be a
    diagram with \( \cat K \in \Phi \). We have a natural isomorphism
    \begin{equation}
        \label{eq:phi.irreducible}
        \Pp_\Phi(\cat C)\big(\yy(C),\colim_{k\in \cat K} Ek\big) 
            \iso \colim_{k \in \cat K} \Pp_\Phi(\cat C)\big(\yy(C),Ek\big),
    \end{equation}
    and dually, for a diagram \( F \colon \cat K \to \Ll_\Phi(\cat C) \),
    \begin{equation}
        \label{eq:phi.coirreducible}
        \Ll_\Phi(\cat C)\big(\lim_{k\in \cat K} Fk,\yy(C)\big) 
            \iso \colim_{k \in \cat K} \Ll_\Phi(\cat C)\big(Fk,\yy(C)\big).
    \end{equation}
\end{lemma}
\begin{proof}
    We have
    \begin{align*}
        \Pp_\Phi(\cat C)&\big(\yy(C),\colim_{k\in \cat K} Ek\big) \\
        &\iso \CAT (\cat C^\op,\Set)\big(\cat C(-,C),
                                        \colim_{k \in \cat K} Ek\big)
        \\
        &\iso \big(\colim_{k \in \cat K} Ek\big)(C)
        &\text{Yoneda lemma,} \\
        &\iso \colim_{k \in \cat K} \big((Ek)C\big)
        &\text{componentwise colimits,} \\
        &\iso \colim_{k \in \cat K} \CAT(\cat C^\op,\Set)\big(\cat C(-,C),Ek\big)
        &\text{Yoneda lemma,} \\
        &\iso \colim_{k \in \cat K} \Pp_\Phi(\cat C)\big(\yy(C),Ek\big).
    \end{align*}
\end{proof}

This leads to the following formulas for the sets of morphisms (hom-sets), based on the
observation that representable functors preserve limits.
\begin{corollary}
    \label{cor:formula}
    Let \( \Phi \) be a class of small categories. If \( \cat J, \cat K \in
    \Phi\), and \( F \colon \cat J \to \cat C \), \( G \colon \cat K \to
    \Pp_\Phi(\cat C)\), then
    \begin{equation}
        \label{eq:main.formula.col}
        \Pp_\Phi(\cat C)\big(\colim_{j\in \cat J} Fj, 
                                \colim_{k \in \cat K} Gk\big)
            \iso \lim_{j \in \cat J} 
                 \colim_{k \in \cat K} \Pp_\Phi(\cat C)(Fj,Gk),
    \end{equation}
    and dually, if \(H:\cat K\to \Ll_\Phi(\cat C)\), then
    \begin{equation}
        \label{eq:main.formula.lim}
        \Ll_\Phi(\cat C)\big(\lim_{k \in \cat K}Hk,\lim_{j\in \cat J}Fj\big)
        \iso \lim_{j\in \cat J} \colim_{k\in \cat K} \Ll_\Phi(\cat C)(Hk,Fj),
    \end{equation}
    where we identify
    an object of \(\cat C\) with its image in \( \Pp_\Phi(\cat C) \) and \(
    \Ll_\Phi(\cat C) \).
\end{corollary}

Alternatively, one may constuct \( \Pp_\Phi(\cat C) \) , and, dually, \(
\Ll_\Phi(\cat C) \), via transfinite induction \cite{Kel82, AK88}, by iteratively
adjoining (co)limits of diagrams with domain in \( \Phi \), and taking unions at
limit ordinals. In certain important cases, such as those of small (or finite)
(co)limit or (co)product completions (see below), the induction stabilises after
only one step. 

Therefore, if \( \Phi \) is a class of small categories such that the
transfinite construction converges in one step, every object in \( \Pp_\Phi(\cat
C)\) is obtained as the \(\Phi\)-colimit of a diagram in \(\cat C\), from which
we obtain the  following characterisation of the \( \Phi \)-(co)limit completion
of \( \cat C \); \( \Pp_\Phi(\cat C) \) consists of
\begin{itemize}[label=--]
    \item 
        diagrams \( F \colon \cat J \to \cat C\) with \( \cat J \in \Phi \) as
        objects, 
    \item 
        hom-sets given by the formula\footnote{See \cite[Section~1]{Tho84},
        and compare with \eqref{eq:main.formula.col}.}
        \begin{equation}
            \label{eq:free.colim.form}
            \Pp_\Phi(\cat C)(F,G) = \lim_{j \in \cat J} 
                             \colim_{k \in \cat K} \cat C(Fj,Gk)
        \end{equation}
        for diagrams \( F \colon \cat J \to \cat C \), \( G \colon \cat K \to
        \cat C \) with \( \cat J, \cat K \in \Phi \).
\end{itemize}

Dually, in case every object in \( \Ll_\Phi(\cat C)\) is obtained as the
\(\Phi\)-limit of a diagram in \(\cat C\),  the free limit completion of a
category \( \cat C \) is given by \( \Ll_\Phi(\cat C) = \Pp_\Phi(\cat C^\op)^\op
\). Explicitly, it consists of
\begin{itemize}[label=--]
    \item 
        diagrams \( F \colon \cat J \to \cat C\) with \( \cat J \in \Phi \) as
        objects,
    \item
        hom-sets given by the formula
        \begin{equation}
            \label{eq:free.lim.form}
            \Ll_\Phi(\cat C)(F,G) = \lim_{k \in \cat K} 
                             \colim_{j \in \cat J} \cat C(Fj,Gk)
        \end{equation}
        for diagrams \( F \colon \cat J \to \cat C \), \( G \colon \cat K \to
        \cat C \) with \( \cat J, \cat K \in \Phi \).
\end{itemize}

Such a characterisation is appropriate, for example, when $\Phi$ consists of the
class of all small (resp. finite) discrete categories, yielding small (resp.
finite) coproduct and product completions, or if $\Phi$ consists of the class of
all small (resp. finite) categories, yielding small (resp. finite) colimit and
limit completions. 

\paragraph*{Free coproduct completion:}
    If \( \Phi \) is the class of discrete small categories (sets), then \(
    \PhiColim \) is the 2-category of categories with coproducts, 
    coproduct-preserving functors and all natural transformations. In this case, we write \( \Fam =
    \Pp_\Phi \). 

    We can explicitly describe the objects of \( \Fam(\cat C) \) -- these are
    given by set-indexed families of objects \( (X_i)_{i \in I} \), with \( X_i
    \in \cat C \). Using the representation coming out of Corollary
    \ref{cor:formula}, we can also describe the hom-sets of morphisms from
    \((X_i)_{i \in I} \) to \((Y_j)_{j \in J} \) as 
    \begin{equation*}
        \prod_{i\in I}\sum_{j\in J} \cat C(X_i, Y_j).
    \end{equation*}
    There is a wealth of literature studying free coproduct completions and
    their properties. For instance, we refer the reader to \cite{CLW93, AR20}, \cite[Chapter~6]{BJ01}, \cite[Section~7]{PL23b}, and \cite{LV24b}.

\paragraph*{Free (co)limit completion:}
    When \( \Phi \) consists of all small categories, \( \PhiColim \) is the
    2-category of categories with small colimits and small-colimit preserving
    functors. 
    
    Given a category \( \cat C \), its \textit{free colimit completion} \(
    \Pp(\cat C) \) is the full subcategory of \( \CAT(\cat C^\op,\Set) \)
    consisting of the \textit{essentially small} or \textit{accessible} functors
    \cite{Kel82}. When \( \cat C \) is itself essentially small, we have \(
    \Pp(\cat C) \eqv \CAT(\cat C^\op,\Set) \).
    
    Alternatively, as noted above, we can characterise \(\Pp(\cat C)\) as the
    category with diagrams \(F \colon \cat J\to \cat C\) with \( \cat J \) small
    as objects and homsets
    \begin{equation*}
        \Pp(\cat C)(F, G)
            = \lim_{j\in \cat J} \colim_{k\in \cat K}\cat C(Fj, Gk)
    \end{equation*}
    for diagrams \( F \colon \cat J \to \cat C \), \( G \colon \cat K \to
    \cat C \) with \( \cat J, \cat K \) small.

\paragraph*{Free finite limit completion:}
    We consider the class \( \Phi = \fin \) of all finite categories, in which
    case \( \PhiLim \) is the 2-category of categories with finite limits and
    the functors that preserve them. We denote the free finite limit completion
    pseudomonad by \( \Lfin \). 

    For any given category \(\cat C\), the category \( \Lfin(\cat C) \) also
    admits a description as a category of diagrams, similar to \( \Ll(\cat C) \).

\paragraph*{Free pullback completion:}
    We consider the class \( \Phi = \pb \) consisting of a single element,
    the cospan category: \( \cdot \rightarrow \cdot \leftarrow \cdot \)

    The 2-category \( \PhiLim \) is the 2-category of categories with pullbacks
    and pullback preserving functors between them, and we denote the free
    pullback completion pseudomonad by \( \Lpb \). 

    Unlike previous examples, not every object in \( \Lpb(\cat C) \) can be
    obtained by taking the pullback of a diagram in \( \Lpb(\cat C) \) of objects
    in the essential image of \( \yy \colon \cat C \to \Lpb(\cat C) \), so we
    cannot recover any formulae analogous to \eqref{eq:free.lim.form}; we refer
    the interested reader to \cite[Section 7]{AK88} for further details.

\section{Three pseudomonads}
\label{sect:pseudomonad}

Let \( \Tt \) be a pseudomonad on \( \CAT \). We consider the following instance
of the main result from \cite{Wal19}:

\begin{lemma}
    The following are equivalent:
    \begin{enumerate}[label=\textnormal{(\roman*)}]
        \item 
            \( \Fam \) lifts to a (lax idempotent) pseudomonad \( \FamT \) on \(
            \TPsAlg \).
        \item  
            There exists a pseudodistributive law \( \delta \colon \Tt \circ
            \Fam \to \Fam \circ \Tt \).
    \end{enumerate}
\end{lemma}
\begin{proof}
    Since \( \Fam \) is a lax idempotent pseudomonad \cite{Koc95}, we may
    instantiate \cite[Theorem 35]{Wal19} with \( \Pp = \Fam \).
\end{proof}

In the presence of a pseudodistributive law \( \delta \colon \Tt \circ \Fam \to
\Fam \circ \Tt \), the composite \( \Fam \circ \Tt \) also has the structure of
a pseudomonad on \( \CAT \) \cite[Section~5]{Mar99}. We also recall the
following result from \cite[Section~6]{Mar04}:

\begin{lemma}
    \label{lem:bieq}
    We have a biequivalence \( \FamT\dash\PsAlg \eqv (\Fam \circ \Tt) \dash
    \PsAlg \).
\end{lemma}

In \cite{Mar04} we also find a description of the \( \FamT \)-pseudoalgebras;
they are the categories \( \cat C \) together with
\begin{itemize}[label=--]
    \item
        a \( \Tt \)-pseudoalgebra structure \( \Lambda \colon \Tt(\cat C) \to
        \cat C \) on \( \cat C \),
    \item
        a \( \Fam \)-pseudoalgebra structure \( \sum \colon \Fam(\cat C) \to
        \cat C \) on \( \cat C \) -- in other words, \( \cat C \) is a category
        with coproducts, 
    \item
        The coproduct functor \( \sum \colon \Fam(\cat C) \to \cat C \) lifts to a
        \(\Tt\)-pseudomorphism.
\end{itemize}

Moreover, a \( \FamT \)-pseudomorphism \( F \colon \cat C \to \cat D \) is a
functor \(F\) that preserves coproducts and is a \(\Tt\)-pseudomorphism in a
compatible way (up to natural isomorphism).

Our work focuses on pseudomonads \( \Tt \) that are free \( \Phi \)-limit
completions for a class \( \Phi \) of small categories.  For simplicity, we
introduce the following terminology: 

\begin{definition}
    For a class \( \Phi \) of small categories, we say that the \( (\Fam \circ
    \Ll_\Phi) \)-pseudoalgebras are the \textit{\( \Phi \)-coproduct
    distributive categories}.
\end{definition}

In this setting, we have
the following result.

\begin{lemma}
    \label{lem:fam.lifts}
    For a class \( \Phi \) of small categories, \( \Fam \) lifts to a
    pseudomonad \( \Fam_{\Ll_\Phi} \) on \( \Phi \dash \Lim \). Consequently, $
    \Fam_{\Ll_\Phi} \dash \PsAlg$ is biequivalent to $(\Fam \circ \Ll_\Phi)
    \dash \PsAlg$.
\end{lemma}
\begin{proof}
    Since \( \Fam(\cat C) \) has whichever \( \Phi \)-limits that \( \cat C \)
    has and \(\Fam(F)\) is $\Phi$-limit preserving whenever $F$ is
    \cite[Section~4]{Gra66}, we conclude that \( \Fam \) lifts to an
    endo-2-functor on \( \Phi \dash \Lim \), and \( \yy \colon \cat C \to
    \Fam(\cat C) \) preserves \( \Phi \)-limits. Moreover, since we have a fully
    faithful adjoint string
    \begin{equation*}
        \Fam \cdot \yy \adj \mathfrak{m} \adj \yy \cdot \Fam,
    \end{equation*}
    we note that, in particular, \( \mathfrak{m} \) is a right adjoint, and
    therefore preserves \( \Phi \)-limits.
\end{proof}

In \cite{LV24}, we study the pseudodistributive laws of the free product
completion pseudomonad \( \Ll_\Set = \Fam\big((-)^\op\big)^\op \) and the free
finite product completion pseudomonad \( \Ll_\finSet =
\FinFam\big((-)^\op\big)^\op \) over \( \Fam \), taking \( \Set \) (\( \finSet
\)) to be the class of small (finite), discrete categories. The composite
pseudomonads \( \Dist = \Fam \circ \Ll_\Set \) and \( \Fam \circ \Ll_\finSet \)
are the pseudomonads whose pseudoalgebras are the doubly-infinitary distributive
categories and infinitary distributive categories, respectively. Under the
terminology we introduced, these are the \textit{product-coproduct distributive}
categories and the \textit{finite product-coproduct distributive} categories. In
the current work, we shall see that:
\begin{itemize}
    \item 
        $\Phi$-coproduct distributive categories are infinitary lextensive
        categories, for the class $\Phi$ of finite categories (which corresponds
        to distributivity of finite limits over coproducts); 
    \item
        $\Phi$-coproduct distributive categories are the infinitary extensive
        categories with pullbacks, for the singleton class $\Phi$ consisting of
        the cospan category \( \cdot \rightarrow \cdot \leftarrow \cdot \)
        (which corresponds to distributivity of pullbacks over coproducts); 
    \item 
        $\Phi$-coproduct distributive categories are the infinitary lextensive
        categories that are doubly-infinitary distributive as well, for the
        class $\Phi$ of all small categories (which corresponds to
        distributivity of limits over coproducts).
\end{itemize}

\subsection{Infinitary lextensive categories}
\label{subsect:free.lext}

We recall that a category with small coproducts \( \cat C \) is
\textit{infinitary extensive} if it has pullbacks along coproduct inclusions,
and if the coproducts are \textit{disjoint} and \textit{pullback-stable}. This
can be expressed in three conditions:
\begin{enumerate}[label=(\alph*)]
    \item
        \label{enum:disjoint}
        for every pair of objects \( A, B \in \cat C \), we have a pullback
        diagram:
        \begin{equation*}
            \begin{tikzcd}
                0 \ar[r] \ar[d] 
                    \ar[rd,"\ulcorner"{very near start,rotate=180},phantom]
                & A \ar[d] \\ B \ar[r] & A+B
            \end{tikzcd}
        \end{equation*}
    \item
        \label{enum:pb-stab-1}
        for each morphism \( f \colon Y \to \sum_{i \in I} X_i\), if we take 
        pullbacks along the coproduct inclusions \( X_i
        \xrightarrow{\iota_i} \sum_{i \in I} X_i \), 
        \begin{equation*}
            \begin{tikzcd}
                Y_i \ar[r,"\iota_i"] \ar[d] 
                    \ar[rd,"\ulcorner"{very near start,rotate=180},phantom]
                    & Y \ar[d,"f"] \\
                X_i \ar[r,"\iota_i",swap] & \sum_{i \in I} X_i
            \end{tikzcd}
        \end{equation*}
        we have that \( Y_i \xrightarrow{\iota_i} Y \) form a coproduct
        diagram as well, and
    \item 
        \label{enum:pb-stab-2}
        for every family \( (f_i \colon Y_i \to X_i)_{i \in I} \) of
        morphisms, the following commutative square
        \begin{equation*}
            \begin{tikzcd}
                Y_i \ar[r,"\iota_i"] \ar[d, "f_i"'] 
                    \ar[rd,"\ulcorner"{very near start,rotate=180},phantom]
                    & \sum_{i \in I} Y_i \ar[d,"\sum_{i\in I}f_i"] \\
                X_i \ar[r,"\iota_i",swap] & \sum_{i \in I} X_i
            \end{tikzcd}
        \end{equation*}
        is a pullback diagram.
\end{enumerate}

We also make use of the following notation: if \( \cat C \) is a category with
coproducts and a terminal object \( \terminal \), we let \( - \pt \terminal
\colon \Set \to \cat C \) be the functor left adjoint to \( \cat C(\terminal, -)
\colon \cat C \to \Set \). We highlight that if \( \cat C \) has a terminal
object \( \terminal \), then so does \( \Fam(\cat C) \), so we have a functor \(
- \pt \terminal \colon \Set \to \Fam(\cat C) \).

The following result, appearing in \cite{CV04} and \cite{PL23b}, is an important
step in the characterization of the \( \Fam_{\Ll_\fin} \)-pseudoalgebras:
\begin{lemma}
    \label{lem:lex}
    Let \( \cat C \) be a category with finite limits and coproducts. Then the
    following are equivalent:
    \begin{enumerate}[label=\textnormal{(\roman*)}]
        \item 
            \label{enum:c.lex}
            \( \cat C \) is infinitary lextensive;
        \item
            \label{enum:coprod.c.lex}
            \( \sum \colon \Fam(\cat C) \to \cat C \) preserves finite limits.
    \end{enumerate}
\end{lemma}
\begin{proof}
    For an infinitary lextensive \( \cat C \), \cite[Lemma~7.1]{PL23b}
    guarantees that we have an equivalence \( \Fam(\cat C) \eqv \big(\cat C
    \comma (-\pt \terminal )\big) \), and that the projection \(\big(\cat C
    \comma (-\pt \terminal)\big) \to \cat C \) preserves finite limits.
    Moreover, we also establish that the composite
    \begin{equation*}
        \begin{tikzcd}
            \Fam(\cat C) \ar[r,"\eqv"] 
                & \big(\cat C \comma (- \pt \terminal)\big) \ar[r]
                & \cat C
        \end{tikzcd}
    \end{equation*}
    corresponds to the coproduct functor  \( \Fam(\cat C)\to \cat C\). This
    shows that  \ref{enum:c.lex} \( \implies \)  \ref{enum:coprod.c.lex}.

    Now, if we assume \ref{enum:coprod.c.lex}, it follows in particular that \(
    \sum \) preserves pullbacks. So, we consider the following pullback diagrams
    in \( \Fam(\cat C) \)
    \begin{equation}
        \label{eq:in.diag}
        \begin{aligned}
            \begin{tikzcd}
                \emptyset \ar[r] \ar[d] 
                        \ar[rd,"\ulcorner"{rotate=180,very near start},phantom]
                & A_0 \ar[d] \\
                A_1 \ar[r] & (A_i)_{i \in \{0,1\}}
            \end{tikzcd}
            &
            \begin{tikzcd}
                (Y_i)_{i \in I} \ar[r] \ar[d] 
                        \ar[rd,"\ulcorner"{rotate=180,very near start},phantom]
                & Y \ar[d,"f"] \\
                (X_i)_{i\in I} \ar[r] & \sum_{i \in I} X_i
            \end{tikzcd}
            &
            \begin{tikzcd}
                V_j \ar[r] \ar[d] 
                        \ar[rd,"\ulcorner"{rotate=180,very near start},phantom]
                & (V_j)_{j \in J} \ar[d] \\
                W_j \ar[r] & (W_j)_{j \in J}
            \end{tikzcd}
        \end{aligned}
    \end{equation} 
    for objects \( A_0, A_1 \in \cat C \), a morphism \( f \colon Y \to \sum_{i
    \in I} X_i\) in \( \cat C \), and a family of morphisms \( (g_j \colon V_j
    \to W_j)_{j \in J} \) in \( \cat C \).

    Since the coproduct functor preserves pullbacks, it can be composed with
    each diagram \eqref{eq:in.diag} to respectively obtain the pullback diagrams
    in \ref{enum:disjoint}, \ref{enum:pb-stab-1} and \ref{enum:pb-stab-2}.
    Hence, we witness the infinitary extensivity of \( \cat C \), thereby
    confirming that \ref{enum:coprod.c.lex} \( \implies \) \ref{enum:c.lex}.
\end{proof}

Now, by Lemma \ref{lem:bieq} and the description for \( \Fam_{\Lfin}
\)-pseudoalgebras, we conclude, as a corollary, that:

\begin{theorem}
    \label{thm:inf.lext}
    The 2-category \( (\Fam \circ \Lfin) \dash \PsAlg \) consists of
    infinitary lextensive categories, and functors preserving coproducts and
    finite limits.
\end{theorem}

\subsection{Infinitary extensive categories with pullbacks}
\label{subsect:free.ext}

We can still obtain results analogous to Lemma
\ref{lem:lex} even in the absence of terminal objects.

\begin{lemma}
    \label{lem:ext}
    Let \( \cat C \) be a category with coproducts and pullbacks. The following
    are equivalent:
    \begin{enumerate}[label=\textnormal{(\roman*)}]
        \item 
            \label{enum:coprod.pb}
            The coproduct functor \( \sum \colon \Fam(\cat C) \to \cat C \)
            preserves pullbacks.
        \item
            \label{enum:ext}
            \( \cat C \) is infinitary extensive.
    \end{enumerate}
\end{lemma}

\begin{proof}
    If \( \cat C \) is infinitary extensive and has pullbacks, then \( \cat C
    \comma X \) is infinitary lextensive for all objects \(X\). Thus, we may
    apply Lemma~\ref{lem:lex} to conclude that
    \begin{equation*}
        \begin{tikzcd}
            \Fam(\cat C) \comma X \eqv \Fam(\cat C \comma X)     
            \ar[r,"\sum"] & \cat C \comma X
        \end{tikzcd}
    \end{equation*}
    preserves finite limits. Since \( \Fam(\cat C) \) is infinitary extensive, we
    have
    \begin{equation*}
        \Fam(\cat C) \comma (X_i)_{i \in I} 
            \eqv \prod_{i \in I} \Fam(\cat C) \comma X_i,
    \end{equation*}
    and a product of finite limit preserving functors preserves finite limits as
    well. Thus, we deduce that \( \sum \colon \Fam(\cat C) \to \cat C \)
    preserves pullbacks, confirming that \ref{enum:ext} \( \implies \)
    \ref{enum:coprod.pb}.

    Conversely, if \( \sum \colon \Fam(\cat C) \to \cat C \) preserves
    pullbacks, we follow the same argument used for Lemma~\ref{lem:lex}: we
    compose the coproduct functor with each of the diagrams \eqref{eq:in.diag}
    to respectively obtain \ref{enum:disjoint}, \ref{enum:pb-stab-1} and
    \ref{enum:pb-stab-2}, exhibiting infinitary extensiveness. This proves that
    \ref{enum:coprod.pb} \( \implies \) \ref{enum:ext}. 
\end{proof}
   
As a consequence, by Lemma \ref{lem:bieq} and the description of \( \Fam_{\Lpb}\)-pseudoalgebras, we conclude that:

\begin{theorem}
    \label{thm:inf.ext.pb}
    The 2-category \( (\Fam \circ \Ll_\pb)\dash\PsAlg \) consists of infinitary
    extensive categories with pullbacks, and functors which preserve coproducts
    and pullbacks.
\end{theorem}

\subsection{Doubly infinitary lextensive categories}
\label{subsect:free.dbl.inf.lext}

Inspired by the terminology of \cite{LV24}, we call the \( (\Fam \circ \Ll)
\)-pseudoalgebras \textit{doubly-infinitary lextensive} categories.

\begin{theorem}
    \label{thm:db.inf.lex}
    Let \( \cat C \) be a category with coproducts and limits. The following are
    equivalent:
    \begin{enumerate}[label=\textnormal{(\roman*)}]
        \item 
            \label{enum:coprod.lim}
            The coproduct functor \( \sum \colon \Fam(\cat C) \to \cat C \)
            preserves limits;
        \item
            \label{enum:db.inf.lext}
            \( \cat C \) is doubly infinitary extensive;
        \item 
            \label{enum:db.inf.dist.lex}
            \( \cat C \) is lextensive and doubly infinitary distributive.
    \end{enumerate}
\end{theorem}

\begin{proof} 
    We have the equivalence \ref{enum:coprod.lim} \( \iff \)
    \ref{enum:db.inf.lext} by definition.

    The equivalence \ref{enum:db.inf.dist.lex} \( \iff \) \ref{enum:db.inf.lext}
    follows by Lemma~\ref{lem:ext} and \cite[Lemma~3.1]{LV24}. We use the basic
    facts that any limit can be obtained via pullbacks and arbitrary products,
    and that infinitary extensive categories with products are, in particular,
    infinitary distributive (see \cite[Proposition 4.5]{CLW93}).
\end{proof}

\section{Exponentiability in freely generated structures}
\label{sect:expn}

The purpose of this section is to study the exponentiable objects of the free
completions \( \Fam\big(\Lfin(\cat C)\big) \) and \( \Fam\big(\Ll(\cat C)\big)
\), which constitute the main results of this work. Aiming for a self-contained
account of exponentiability, we begin by recalling the definition of
\textit{exponentiable object}, as well as some elementary properties.

In order to fix notation, we recall that an object \(E\) in a category \(\cat C
\) with finite products is \textit{exponentiable} at \(X\) if there exists an
object \( E \expn X \) and a natural isomorphism
\begin{equation}
    \label{eq:expn.hom.iso}
    \cat C(- \times E, X) \iso \cat C(-,E\expn X).
\end{equation}
We say that \(E\) is \textit{exponentiable} if \eqref{eq:expn.hom.iso} holds
naturally for every object \(X\) in \(\cat C \).

We revisit the following elementary observation about exponentiable objects used
in \cite[Remark~1]{LV24}:
\begin{lemma}
    \label{lem:exp.lims}
    Let \( \cat C \) be a category with finite products and \( \cat J \)-limits
    for a small category \( \cat J \). If \( F \colon \cat J \to \cat C \) is a
    diagram, and \( E \) is an object such that \( E \) is exponentiable at \(
    Fj \) for each \(j\) in \( \cat J\), then \( E \) is exponentiable at \(
    \lim_{j \in \cat J} F \) and
    \begin{equation*}
        E \expn \lim_{j \in \cat J} Fj \iso \lim_{j \in \cat J} (E \expn Fj)
    \end{equation*}
\end{lemma}
\begin{proof}
    For each object \(A\), we have a natural isomorphism
    \begin{equation*}
        \cat C(A\times E,\lim_{j \in \cat J}Fj)
            \iso \lim_{j\in \cat J} \cat C(A\times E, F_j)
            \iso \lim_{j\in \cat J} \cat C(A, E\expn F_j)
            \iso \cat C\big(A,\lim_{j \in \cat J} (E\expn F_j)\big)
    \end{equation*}
    as desired.
\end{proof}

We recall from \cite[Definition 6.1.3]{BJ01} that an object \(A\) of a category
\( \cat C \) is \textit{connected} if the hom-functor \( \cat C(A,-) \) preserves
coproducts. It is an immediate consequence of Lemma \ref{lem:phi.irreducible}
that the objects in the essential image of \( \yy \colon \cat C \to \Fam(\cat C)
\) are precisely the connected objects in \( \Fam(\cat C) \). We confirm that an
analogous characterization is available for the internal hom-functor:

\begin{lemma}
    \label{lem:expn.coprod}
    If \( \cat C \) is a category with finite products, and \( C \) is an
    exponentiable object in \(\Fam(\cat C) \), then the following are equivalent:
    \begin{enumerate}[label=\textnormal{(\roman*)}]
        \item
            \label{enum:conn}
            \(  C \) is connected.
        \item
            \label{enum:coprod.presv}
            \( C \expn - \) preserves coproducts.
    \end{enumerate}
\end{lemma}
\begin{proof}
    Let \( (A_i)_{i\in I} \) be a family of objects in \(\cat C\), and let \(
    (X_j)_{j \in J} \) be a family of objects in \(\Fam(\cat C) \). If \( \cat C
    \) is connected, then we have natural isomorphisms
    \begin{align*}
        \Fam(\cat C)&\Big((A_i)_{i\in I} \times C, \sum_{j\in J} X_j\Big)\\
        &\iso \Fam(\cat C)\Big((A_i \times C)_{i \in I}, \sum_{j \in J} X_j\Big)
        &\text{products in }\Fam(\cat C),\\
        &\iso \prod_{i \in I} \sum_{j \in J} \Fam(\cat C)(A_i \times C, X_j)
        &\eqref{eq:main.formula.col},\\
        &\iso \prod_{i \in I} \sum_{j \in J} \Fam(\cat C)(A_i, C \expn X_j)
        &C \text{ exponentiable},\\
        &\iso \Fam(\cat C)\Big((A_i)_{i \in I}, \sum_{j \in J} (C \expn X_j)\Big)
        &\eqref{eq:main.formula.col}. \\
    \end{align*}
    Hence, we conclude that
    \begin{equation*}
        \sum_{j \in J} (C \expn X_j) \iso C \expn \sum_{j \in J} X_j,
    \end{equation*}
    which confirms that \ref{enum:conn} \( \implies \) \ref{enum:coprod.presv}.

    Conversely, if \( C \expn -\) preserves coproducts, then for a family \(
    (X_j)_{j\in J} \) of objects in \(\Fam(\cat C) \), we have
    \begin{align*}
        \Fam(\cat C)\Big(C,\sum_{j \in J} X_j\Big)
        &\iso \Fam(\cat C)\Big(\terminal ,C \expn \sum_{j \in J} X_j\Big) 
        & C\text{ exponentiable},\\
        &\iso \Fam(\cat C)\Big(\terminal ,\sum_{j \in J} C \expn X_j\Big) 
        & \text{by hypothesis }\ref{enum:coprod.presv},\\
        &\iso \sum_{j \in J} \Fam(\cat C)(\terminal ,C \expn X_j)
        & \text{terminal connected},\\
        &\iso \sum_{j \in J} \Fam(\cat C)(C, X_j),
    \end{align*}
    hence, we conclude that \ref{enum:coprod.presv} \( \implies \)
    \ref{enum:conn}.
\end{proof}

Let \( \Phi \) be a class of small categories that includes all finite, discrete
categories, so that every \( \Ll_\Phi \)-pseudoalgebra has finite products. For
the sake of succinctness, we say that an object of \( \Fam(\Ll_\Phi(\cat C)) \) is a
\textit{generator} if it is in the essential image of the inclusion \( \cat C
\to \Fam(\Ll_\Phi(\cat C)) \).

We will give an inductive perspective on exponentials in \(
\Fam\big(\Ll_\Phi(\cat C)\big) \), and the following result is the cornerstone
for our development (see \cite[Remark 1]{LV24}):

\begin{lemma}
    \label{lem:exp.gen}
    If \( X \) is a generator and \(D\) is connected in \(
    \Fam\big(\Ll_\Phi(\cat C)\big) \), then \(D\) is exponentiable at \(X\) and
    we have 
    \begin{equation*}
        D \expn X \iso X + \hat{\cat C}(D,X)\pt \terminal
    \end{equation*}
    where \( \hat{\cat C} = \Fam\big(\Ll_\Phi(\cat C)\big) \).
\end{lemma}
\begin{proof}
    Let \( (E_i)_{i \in I} \) be a family of objects in \( \Ll_\Phi(\cat C) \). 
    We have natural isomorphisms
    \begin{align*}
        \hat{\cat C}&\big((E_i)_{i \in I} \times D, X\big)\\
        &\iso \hat{\cat C}\big((E_i \times D)_{i \in I}, X\big) 
            & \text{products in } \hat{\cat C} \\
        &\iso \prod_{i \in I} \hat{\cat C}(E_i \times D, X) 
        & \hat{\cat C}(-,X)\text{ preserves products}, \\
        &\iso \prod_{i \in I} \Lfin (\cat C)(E_i \times D, X) 
        &\text{full faithfulness}, \\
        & \iso \prod_{i \in I} \Lfin (\cat C)(E_i,X)  + \Lfin (\cat C)(D,X) 
        &\eqref{eq:main.formula.lim}, \\
        &\iso  \prod_{i \in I} \hat{\cat C}(E_i,  X) + \hat{\cat C}(D,X) 
        &\text{full faithfulness}, \\
        &\iso 
        \hat{\cat C}\big((E_i)_{i \in I},  X+\hat{\cat C}(D,X) \pt \terminal \big)
        & \eqref{eq:main.formula.col}
    \end{align*}
\end{proof}

\subsection{Exponentials for free doubly infinitary lextensive categories}

Having reviewed the elementary properties of exponentiable objects, we proceed
to prove our main result on exponentiability of the objects of freely generated
doubly-infinitary lextensive categories:
    
\begin{theorem}
    \label{thm:main.two}
    The category \( \Fam\big(\Ll(\cat C)\big) \) is cartesian closed.
\end{theorem}

\begin{proof}
    First, we note that connected objects are exponentiable: 
    \begin{itemize}[label=--]
        \item 
            By Lemma~\ref{lem:exp.gen}, we have that any connected object in \(
            \Fam\big(\Ll(\cat C)\big) \) is exponentiable at the generators. 
        \item
            Any connected object is a limit of generators, so by
            Lemma~\ref{lem:exp.lims} we conclude that connected objects are
            exponentiable at any connected object in \( \Fam\big(\Ll(\cat
            C)\big) \). 
        \item
            Since any object in \( \Fam\big(\Ll(\cat C)\big) \) is a coproduct
            of connected objects, we simply apply Lemma~\ref{lem:expn.coprod} to
            deduce our claim.
    \end{itemize}

    Now, let \( (E_i)_{i \in I} \) and \( (D_j)_{j \in J} \) be families of
    objects in \( \Ll(\cat C) \), and \(X\) any object in \( \hat{\cat C} \).
    We have natural isomorphisms
    \begin{align*}
        \hat{\cat C}&\big((E_i)_{i \in I} \times (D_j)_{j \in J}, X\big) \\
        &\iso 
        \hat{\cat C}\big((E_i \times D_j)_{(i,j) \in I\times J}, X\big)
        &\text{binary products in }\hat{\cat C}, \\
        &\iso 
        \prod_{i \in I} \prod_{j \in J} \hat{\cat C}(E_i \times D_j, X)
        &\hat{\cat C}(-,X) \text{ preserves limits}, \\
        &\iso 
        \prod_{i \in I} \prod_{j \in J} \hat{\cat C}(E_i, D_j \expn X)
        &D_j \text{ connected (exponentiable)}, \\
        &\iso 
        \prod_{i \in I} \hat{\cat C}\Big(E_i, \prod_{j \in J} (D_j \expn X)\Big)
        \\
        &\iso 
        \hat{\cat C}\Big((E_i)_{i\in I}, \prod_{j \in J} (D_j \expn X) \Big)
    \end{align*}
    Thus, we obtain
    \begin{equation*}
        (D_j)_{j \in J} \expn X \iso \prod_{j \in J} (D_j \expn X),
    \end{equation*}
    confirming that coproducts of connected objects are exponentiable. But every
    object in \( \Fam(\Ll(\cat C)) \) is a coproduct of connected objects, hence
    the result follows.
\end{proof}

\subsection{Explicit descriptions of the exponentials}

Let \( (D_j)_{j \in J} \) and \( (E_k \colon \cat A_k \to \cat C)_{k \in K} \)
be families of objects in \( \cat L \), where \( \cat A_k \) is a small category
for each \(k \in K \).

The results of the previous subsection can be used to calculate an explicit
expression for the exponential \( (D_j)_{j \in J} \expn (E_k)_{k \in K} \) in \(
\Fam\big(\Ll(\cat C)\big)\): via Lemmas \ref{lem:exp.lims}--\ref{lem:exp.gen},
one of Theorems \ref{thm:main.one} or \ref{thm:main.two}, and the key ideas of
the proof of \cite[Theorem 2.3]{LV24}, we obtain 
\begin{equation}
    \label{eq:expn-expr}
    (D_j)_{j \in J} \expn (E_k)_{k \in K}
    \iso \Big( \prod_{j \in J} \lim_{l \in \cat A_{f_j^K}} 
                    \Delta_{f,j,l} \Big)_{f \in \Omega}
\end{equation}
where
\begin{align*}
    \Omega &= \prod_{j \in J} \sum_{k \in K} \lim_{l \in \cat A_k} 
                \big(\terminal + \Ll(\cat C)(D_j,E_{k,l})\big),
    \\
    \Delta_{f,j,l} &=
    \begin{cases}
        E_{f_j^K,l} & \text{if } f_j(l) \in \terminal \\
        \terminal  & \text{if } f_j(l) \in \Ll(\cat C)(D_j,E_{f_j^K,l})
    \end{cases} 
\end{align*}
and \( f^K_j \) is the projection of \(f_j\) onto \(K\) for each \( f \in
\Omega \), \(j \in J\).

\begin{remark}
    As long as $\cat C$ has an initial object \( \initial \), the exponentials
    may be given explicitly by 
    \begin{equation*}
        (D_j)_{j \in J} \expn (E_k)_{k \in K}
        \iso \bigg(\prod_{j \in J} \abstrcompldom{\Big(\pi_2\big(f(j)\big)\Big)}  
             \bigg)_{f \in \Omega}
    \end{equation*}
    where
    \begin{equation*}
        \Omega = \prod_{j \in J} \sum_{k \in K} 
                    \big(\Lfin(\cat C)\big)(D_j \times \initial,E_k),
    \end{equation*}
    and $\abstrcompldom{(g)}$ is defined via the following pushout in
    $\Lfin(\cat C)$, by co-extensivity:
    \begin{equation*}
        \begin{tikzcd}
            D_i\times\initial 
                \ar[r, "g"] \ar[d, "\pi_2",swap] 
                \ar[rd,"\ulcorner"{very near end},phantom]
            & E_{\pi_1(f(i))} \ar[d] \\
        \initial \ar[r]        
            & \abstrcompldom{(g)}.       
        \end{tikzcd}
    \end{equation*}
\end{remark}

\subsection{Exponentials for free infinitary lextensive categories}

As we remarked in Section \ref{sect:free.comp}, we have a fully faithful,
finite limit preserving functor 
\begin{equation*}
    \uu \colon \Lfin(\cat C) \to \Ll(\cat C) 
\end{equation*}
for every category \( \cat C \). By studying the fully faithful functor 
\begin{equation*}
    \uuo = \Fam(\uu) \colon \Fam\big(\Lfin(\cat C)\big) 
                     \to \Fam\big(\Ll(\cat C)\big),
\end{equation*}
we can deduce results about exponentiability of objects in \(
\Fam\big(\Lfin(\cat C)\big) \). More precisely, we have

\begin{lemma}
    \label{lem:expn.refl}
    The functor \( \uuo \) reflects exponentials of finite coproducts of
    connected objects.
\end{lemma}

\begin{proof}
    Let \( (D_j)_{j \in J} \) be a \textit{finite} family of objects in \(
    \Lfin(\cat C) \), and let \( (E_k \colon \cat A_k \to \cat C )_{k \in K} \)
    be a family of objects in \( \Lfin(\cat C) \), where \( \cat A_k \) is a
    finite category for each \(k \in K \). Given any object \( X \) in \(
    \Fam\big(\Lfin(\cat C)\big) \), we have 
    \begin{align*}
        \Fam\big(&\Lfin(\cat C)\big)\big(X \times (D_j)_{j \in J}, 
                                    (E_k)_{k \in K}\big) \\
        &\iso \hat{\cat C}\Big(\uuo\big(X \times (D_j)_{j \in J}\big),
                               \uuo((E_k)_{k \in K})\Big)
        & \uuo\text{ fully faithful},\\
        &\iso \hat{\cat C}\Big(\uuo(X) \times \uuo\big((D_j)_{j \in J}\big),
                               \uuo\big((E_k)_{k \in K}\big)\Big)
        &\uuo\text{ preserves binary products} \\
        &\iso \hat{\cat C}\Big(\uuo(X), 
                               \uuo\big((D_j)_{j \in J}\big) 
                                    \expn \uuo\big((E_k)_{k \in K}\big)\Big)
        &\Fam\big(\Ll(\cat C)\big)\text{ is cartesian closed},
    \end{align*}
    where \( \hat{\cat C} = \Fam\big(\Ll(\cat C)\big) \). Moreover, we have
    \begin{align*}
        \uuo\big((D_j)_{j \in J}\big) 
            \expn \uuo\big((E_k)_{k \in K}\big)
        \iso \big(\uu(D_j)\big)_{j \in J} 
                \expn \big(\uu(E_k)\big)_{k \in K},
    \end{align*}
    and calculating the exponential as in \eqref{eq:expn-expr}, we obtain
    \begin{equation*}
        \big(\uu(D_j)\big)_{j \in J} 
            \expn \big(\uu(E_k)\big)_{k \in K}
            \iso \Big( \prod_{j \in J} \lim_{l \in \cat L_{f^K_j}}
                    \uu(\Gamma_{f,j,l}) \Big)_{f \in \Xi}
    \end{equation*}
    where 
    \begin{equation*}
        \Xi = \prod_{j \in J} \sum_{k \in K} \lim_{l \in \Ll_k}
                \big(\terminal + \Lfin(\cat C)(D_j,E_{k,l})\big),
    \end{equation*}
    \begin{equation*}
        \Gamma_{f,j,l} = 
        \begin{cases}
            E_{f_j^K,l} & \text{if } f_j(l) \in \terminal \\
            \terminal & \text{if } f_j(l) \in \Lfin(\cat C)(D_j,E_{f_j^K,l}),
        \end{cases} 
    \end{equation*}
    and \( f^K_j \) is the projection of \(f_j\) onto \(K\) for each \( f \in
    \Xi \), \(j \in J\).

    Now, since \( \uu \) is fully faithful and preserves finite limits, it must
    reflect them as well. Since we are given that \( J \) is finite, as well as
    \( \cat A_k \) for all \( k \in K \), we have 
    \begin{equation*}
        \prod_{j \in J} \lim_{l \in \cat A_{f^K_j}} \uu(\Gamma_{f,j,l})
        \iso \uu\Big( \prod_{j \in J} \lim_{l \in \cat A_{f^K_j}} \Gamma_{f,j,l}
                \Big).
    \end{equation*}
    and thus
    \begin{equation*}
        \uuo\big((D_j)_{j \in J}\big) 
            \expn \uuo\big((E_k)_{k \in K}\big)
        \iso \uuo\Big( \prod_{j \in J} \lim_{l \in \cat L_{f^K_j}}
                    \Gamma_{f,j,l} \Big)_{f \in \Xi}
    \end{equation*}
    so, since \( \uuo \) is fully faithful, we conclude that the exponential \(
    (D_j)_{j \in J} \expn (E_k)_{k \in K} \) in \( \Fam(\Lfin(\cat C)) \) exists and
    \begin{equation*}
        (D_j)_{j \in J} \expn (E_k)_{k \in K}
        \iso  \Big( \prod_{j \in J} \lim_{l \in \cat L_{f^K_j}}
                    \Gamma_{f,j,l} \Big)_{f \in \Xi},
    \end{equation*}
    as desired.
\end{proof}

As an immediate corollary, we obtain our second main result:
\begin{theorem}
    \label{thm:main.one}
    Finite coproducts of connected objects in \( \Fam(\Lfin(\cat C)) \) are
    exponentiable.
\end{theorem}

\section{Examples}
\label{sec:Examples}

In this section, we intend to give a brief discussion on examples of the various
notions of (l)extensive categories arising from the $\Phi$-coproduct
distributive categories discussed herein. More interestingly, we discuss
examples of the \textit{doubly-infinitary lextensive categories} introduced in
\ref{subsect:free.dbl.inf.lext}.

Recall that we consider the notion of doubly-infinitary distributive categories
introduced in \cite{LV24}, and the $2$-functor $\Dist =
\Famm{\Famm{-}^\op}^\op$. By Theorem \ref{thm:db.inf.lex}, doubly-infinitary
lextensive categories are precisely the doubly-infinitary distributive
categories which are also lextensive. With this in mind, we refer the reader to
the examples discussed in \cite{LV24}, and we make some considerations tailored
to our setting.

\subsection{Fundamental examples}
Let $\terminal $ be the terminal category (the category with precisely one
object and the identity morphism), and $\emptyset $ the initial category (the
empty category).

Let $\Phi$ be a class of small categories containing \( \emptyset \). The
category of sets
\begin{equation*}
    \Set \eqv \Fam(\terminal ) \eqv \Fam(\Ll_\Phi(\emptyset ))
\end{equation*}
is the free $\Phi $-coproduct distributive category on the empty category.
Hence, it is the initial object in the 2-category of $\Phi$-coproduct
distributive categories.

Let \( \Phi \) be a class of small categories containing all discrete
categories. Then the category
\begin{equation*}
    \Fam(\Set^\op) \eqv \Fam\big(\Ll(\terminal )\big) 
                   \eqv \Dist(\terminal) 
\end{equation*}
is the free $\Phi$-coproduct distributive category on the terminal category
$\terminal$. As such, $\Fam(\Set^\op)$ is both the free doubly-infinitary
lextensive category, and the free doubly-infinitary distributive category on
$\terminal$. By Theorem~\ref{thm:main.two} (or \cite[Theorem 2.3]{LV24}), we
conclude that $\Fam(\Set^\op)$ (also known as the category of polynomials) is
cartesian closed -- recovering the result of
\cite{DBLP:conf/cie/AltenkirchLS10}.

\subsection{Monadicity and presheaves}

Let \( \Phi \) be a class of categories. The ``\( \Phi \)-coproduct distributivity''
properties can be lifted through functors that create \( \Phi \)-limits and
coproducts. To be precise, we have the following elementary result:
\begin{lemma}
    \label{lem:monadicity-lextensivity}
	Let $G \colon \cat D\to \cat C $ be a functor that creates coproducts and \(
	\Phi \)-limits. If $\cat C $ is a $\Phi$-coproduct distributive category, then
	so is $\cat D$.
\end{lemma}

Since pseudomonadic pseudofunctors create bicategorical products (see, for
instance, \cite{Luc18b, Luc16} for lifting results on the pseudomonad setting,
and \cite{Str80, Str87} and \cite[{3.8}]{Luc18a} for bilimits), we find that:

\begin{lemma}
    \label{lem:products}
    The 2-categories \( \PhiLim \) and \( (\Fam \circ \Ll_\Phi)\dash \PsAlg \)
    have bicategorical products, given by the product of the underlying
    categories. 
    
    More specifically, if \( (C_i)_{i \in I} \) is a family of \( \Phi \)-limit
    complete (\( \Phi \)-coproduct distributive) categories, then
    \begin{equation*}
	   \prod _{i\in I} \cat C_i 
    \end{equation*}
    is \( \Phi \)-limit complete (\( \Phi \)-coproduct distributive).
\end{lemma}

By applying Corollary \ref{lem:monadicity-lextensivity}, we conclude that:
\begin{theorem}
    Let \( \cat J \) be a small category. If $\cat C $ is a $\Phi$-coproduct
    distributive category, then the functor category $\CAT(\cat J , \cat C) $ is
	$\Phi$-coproduct distributive as well. 
\end{theorem}
\begin{proof}
	The result follows from the fact that the restriction/forgetful functor 
    \begin{equation*}
        \CAT(\cat J, \cat C) \to \CAT (\ob \cat J, \cat C) 
            \iso \prod_{j \in \ob \cat J } \cat C 
    \end{equation*}
	creates limits and colimits that exist in $\cat C$, and Lemma \ref{lem:products}.
\end{proof}
As a consequence, if $ \cat A $ is a small category, the presheaf category \(
\CAT(\cat A^\op, \Set) \) is $\Phi$-coproduct distributive, provided that $\Phi$
contains \( \emptyset \). In particular, \( \CAT(\cat A^\op, \Set) \) is
doubly-infinitary extensive.

\subsection{Finite and small bicategorical biproducts}

As remarked in \cite[4.3]{LV24}, the $2$-category of categories with products is
bicategorically semi-additive. This observation also extends to our setting.

Let \( \Phi \) be a class of small categories containing the finite discrete
categories. We note that the $2$-category \( \PhiLim \) of $\Phi$-limit complete
categories is naturally enriched over the $2$-category of symmetric monoidal
categories with the multilinear multicategorical structure. Together with Lemma
\ref{lem:products}, we conclude that the $2$-category \( \PhiLim \) has finite
bicategorical coproducts, which are equivalent to the finite bicategorical
products. In other words:

\begin{lemma}
    The 2-category \( \PhiLim \) has finite bicategorical biproducts.
\end{lemma} 

Moreover, it is clear that the hom-categories in the $2$-category \( \PhiLim \)
are themselves \( \Phi \)-limit complete -- moreover, noting that composition of
\( \Phi \)-limit preserving functors preserve \( \Phi \)-limits componentwise,
we conclude that:

\begin{lemma}
    \label{natural-enrichment}
    The 2-category of  $\Phi$-limit complete categories is naturally enriched
    over itself, with the multilinear multicategorical structure. 
\end{lemma}

If \( \Phi \) contains all (small) discrete categories, then by Lemmas
\ref{lem:products} and \ref{natural-enrichment},  we conclude that the
$2$-category of $\Phi$-limit complete categories has bicategorical coproducts,
which are equivalent to the bicategorical products. This is given as:

\begin{lemma}
    \label{lem:crazy}
    If \( \Phi \) is a class of small categories containing the discrete
    categories, then the $2$-category of $\Phi$-limit complete categories has
    infinite bicategorical biproducts.
\end{lemma}

This allows us to understand freely generated $\Phi$-coproduct distributive
categories over coproducts of categories, as we, for instance, show in
Subsection \ref{sect:discrete-examples}. 

\subsection{Freely generated categorical structures on discrete
categories}
\label{sect:discrete-examples}

Now, we assume $\Phi $ be a class of small categories that contains all the
small (respectively, finite) discrete categories. 

If \( \cat C \) is a small (finite) discrete category, we have $\cat C \eqv
\sum_{c \in \ob \cat C} \terminal $. Since $\Ll _\Phi $ preserves small (finite)
bicategorical biproducts and $ \Ll_\Phi(\terminal) \eqv \Set^\op$, we have that
\begin{equation*}
    \Ll _\Phi(\cat C) 
        \eqv \Ll_\Phi\Big(\sum_{c\in \ob\cat C} \terminal\Big) 
        \eqv \prod_{c\in \ob\cat C} \Ll_\Phi(\terminal) 
        \eqv \prod_{c\in \ob\cat C} \Set^\op
\end{equation*}
by Lemma \ref{lem:crazy}. Therefore:

\begin{theorem}
    If \( \cat C \) is a small discrete category, then 
    \begin{equation*}
        \Fam(\Ll_\Phi(\cat C)) \eqv  \Fam\Big(\prod _{c\in \ob\cat C} \Set ^\op\Big).
    \end{equation*}
\end{theorem}

In particular, this result describes the free doubly-infinitary distributive
categories, and free doubly-infinitary lextensive categories on a small,
discrete category \( \cat C \). 

\subsection{More on doubly-infinitary lextensive categories via free coproduct
completions}

As we showed in Lemma \ref{lem:fam.lifts}, \( \Fam \) lifts to a pseudomonad
\( \Fam_{\Ll_\Phi} \) on \( \Ll_\Phi \dash\PsAlg \). Thus, if a category \(\cat
C \) has \( \Phi \)-limits, then \( \Fam(\cat C) \) has \( \Phi \)-limits as
well, which are preserved by the coproduct \( \mathfrak m \colon \Fam(\Fam(\cat
C)) \to \Fam(\cat C) \). In particular,
\begin{itemize}[label=--]
    \item
        if \(\cat C \) has pullbacks, then \( \Fam(\cat C) \) is infinitary
        extensive with pullbacks,
    \item 
        if \(\cat C \) has finite limits, then \( \Fam(\cat C) \) is infinitary
        lextensive,
    \item 
        if \(\cat C \) has small limits, then \( \Fam(\cat C) \) is
        doubly-infinitary lextensive,
    \item 
        if \(\cat C\) has products, then \(\Fam(\cat C)\) is
        doubly-infinitary distributive by \cite[Example~1]{LV24}.
\end{itemize}

So, even if a category \( \cat C \) with products does not have small limits, we
can still establish that the category \( \Fam(\cat C) \) is doubly-infinitary
distributive, and it is extensive \cite{CLW93} by virtue of being a free
coproduct completion. Hence, if \( \Fam(\cat C) \) has small limits, we conclude
that it is doubly-infinitary lextensive, by Theorem \ref{thm:db.inf.lex}.

Before discussing our examples, we let \( \Conn(\cat C) \) be the full
subcategory of a category \( \cat C \) with coproducts consisting of the
\textit{connected objects} \cite[Definition 6.1.3]{BJ01}. 

We begin by noting that the category \( \Cat \eqv \Fam(\Conn(\Cat))  \) of small
categories is doubly-infinitary lextensive, as it is both doubly-infinitary
distributive and extensive, and \( \Cat \) has small limits. Likewise, we can
prove that the category \( \wCPO \eqv \Fam(\Conn(\wCPO)) \) of \( \omega
\)-complete partial orders is also a doubly-infinitary lextensive category.

Again similarly, the category \(\LocConTopl\) of locally connected topological
spaces and continuous functions is doubly-infinitary lextensive. Indeed, from
\cite[Example~8]{LV24}, we learn that \( \LocConTopl \eqv
\Fam(\Conn(\LocConTopl)) \) is both doubly-infinitary distributive and
extensive, as the free coproduct completion of a category with products.
Moreover, \( \LocConTopl \) is a coreflective subcategory of \( \Topl \)
\cite{Gle63}, therefore, \(\LocConTopl\) has small limits, letting us conclude
that \( \LocConTopl \) is doubly-infinitary lextensive.

\subsection{Doubly-infinitary distributive categories that are not extensive}

As observed in \cite{LV24}, a distributive lattice \( \cat D \) (seen as a
distributive, thin category) is extensive if and only if \( \cat D \eqv
\terminal \), so any non-trivial example of a completely distributive lattice \(
\cat D \) will be doubly-infinitary distributive, but not extensive. 

Another example is the full subcategory \( \Set^2_\bullet \) of \( \Set \times
\Set \) consisting of those pairs of sets that are either both empty, or both
non-empty. Since coproducts and products are calculated componentwise in \(
\Set^2_\bullet \), this category is doubly-infinitary distributive as well, but
it is not extensive.

\subsection{Cartesian closedness vs. doubly-infinitary lextensivity}

The category \( \Fam(\Topl) \) is an example of a doubly-infinitary lextensive
category that is \textit{not} cartesian closed. We note that the category \(
\Topl \) of topological spaces is infinitary distributive, but not cartesian
closed. So, by \cite[Theorem 4.2]{LV24}, we conclude that \( \Fam(\Topl) \) is
not cartesian closed as well. However, \( \Fam(\Topl) \) is doubly-infinitary
lextensive, since \( \Topl \) has small limits.

An example of a cartesian closed category with all coproducts and limits, but
not doubly-infinitary lextensive, is given in \cite[Counter-example 2]{LV24}, the
category of Quasi-Borel spaces.

\section{Epilogue}
\label{sect:epilogue}

Motivated particularly by the insights from \cite{LV24b, LV24}, the present work
explores the distributive properties of limits over coproducts through the lens
of two-dimensional monad theory~\cite{BKP89, Luc16}.

We have demonstrated that the canonical (pseudo)distributivity of pullbacks over
coproducts leads to a pseudomonad whose pseudoalgebras are precisely the
infinitary extensive categories equipped with pullbacks. Similarly, the
distributivity of finite limits over coproducts leads to the notion of a
pseudomonad whose $2$-category of pseudoalgebras is precisely the $2$-category
of infinitary lextensive categories. Finally, we showed that the distributivity
of limits over coproducts leads to the concept of \textit{doubly-infinitary
lextensivity}, characterized as infinitary extensive categories that are also
doubly-infinitary distributive as introduced in \cite{LV24}.

We also studied the exponentiable objects of the free completions \(
\Fam(\Lfin(\cat C)) \) and \( \Fam(\Ll(\cat C)) \), confirming that the latter is
a cartesian closed category for any category \(\cat C \). These free completions
enjoy various other known properties since they end up being the free coproduct
completion of a well-behaved category -- we refer the reader to \cite{CLW93,
AR20, LV24b, LV23, LV24} for further results.

\subsection*{Free finite coproduct completion}

By replacing the free coproduct pseudomonad \( \Fam \) with its
\textit{finite} counterpart \( \FinFam \), we recover nearly all of our
results, provided we make some adaptations to be finitary setting. Namely, we
obtain a pseudodistributive law
\begin{equation*}
    \Ll_\Phi \circ \FinFam \to  \FinFam \circ \Ll_\phi,
\end{equation*} 
for any class \( \Phi \) of \textit{finite} categories, by reworking the proof of
Lemma \ref{lem:fam.lifts}. We then obtain two more characterizations:
\begin{itemize}[label=--]
    \item 
        the \( (\FinFam \circ \Lpb) \)-pseudoalgebras are precisely the
        extensive categories with pullbacks,
    \item
        the \( (\FinFam \circ \Lfin) \)-pseudoalgebras are precisely the
        lextensive categories.
\end{itemize}
Most consequentially, an adaptation of our exponentiability results will
confirm that \( \FinFam(\Lfin(\cat C)) \) is a cartesian closed category
whenever \( \cat C \) is \textit{locally finite}.

\subsection*{Descent theory}

\textit{Effective descent morphisms} \cite{Gir64, JT97} (see also \cite[Sections
3 and 4]{Luc21}) are the backbone of Grothendieck's descent theory \cite{JT94,
Luc18a}, which has significant consequences in various fields~\cite{Moe89,
RT94, BJ97}. Besides their wide range of applications, effective descent
morphisms hold intrinsic interest, as their purpose is the reconstruction of
data over the codomain from given data over the domain, plus some additional
algebraic structure. 

Of particular relevance to the present work are effective descent morphisms of
freely generated categorical structures. For instance,  \cite[Section
4]{Pre24} studied categories of descent data for families of
morphisms \( \phi \colon (X_i)_{i \in I} \to Y \), as well as conditions under
which \( \phi \) is an effective descent morphism in \( \Fam(\cat C) \), provided
that \( \cat C \) has finite limits. Namely, it was shown that all such descent
data is a coproduct of connected descent data, which provided simpler conditions
for a morphism \( \phi \colon (X_i)_{i \in I} \to Y \) to be of effective
descent -- this gives evidence that \( \Fam(\cat C) \) is a good proxy for the
study effective descent morphisms of \( \cat C \). This perspective was
useful in the study of effective descent functors between enriched
categories, establishing precise connections between the work of \cite{Cre99},
\cite[Theorem 9.11]{Luc18a}, \cite{PL23a}, and the work of \cite{RT94, CH04,
CH17}.

Since the free completions \( \Dist(\cat C) \) and \( \Fam(\Ll(\cat C)) \) are
even better behaved categories, enjoying properties such as cartesian
closedness,  an inquiry on whether studying effective descent morphisms in such
free completions seems to be a reasonable avenue for future work.

\subsection*{Non-canonical isomorphisms}

In analogy with \cite[Subsection 5.2]{LV24}, we may use the results of
\cite{Luc19} to prove that a category \( \cat C \) is \( \Phi \)-coproduct
distributive if it has coproducts, \( \Phi \)-limits, and there exists a(ny)
invertible natural isomorphism
\begin{equation*}
    \begin{tikzcd}
        \displaystyle\sum_{x \in \lim\limits_{j \in \cat J} UF} 
        \displaystyle\lim_{j \in \cat J} F_{j,x_j}
        \ar[r,"\iso"]
        & \displaystyle\lim_{j \in \cat J} 
          \displaystyle\sum_{x \in UFj} F_{j,x}
    \end{tikzcd}
\end{equation*}
for every functor \( F \colon \cat J \to \Fam(\cat C) \) with \( \cat J \in \Phi
\), where we let \( U \colon \Fam(\cat C) \to \Set \) be the functor that
outputs the underlying indexing set. 

More generally, if we have a pseudomonad \( \Tt \) on \( \CAT  \) and a
pseudodistributive law \( \delta \colon \Tt \circ \Fam \to \Fam \circ \Tt \),
then for any category \( \cat C \) with coproducts and the structure of a  \(
\Tt \)-pseudoalgebra, the coproduct functor
\begin{equation*}
    \sum \colon \Fam(\cat C) \to \cat C
\end{equation*}
is an oplax \( \Tt \)-morphism by doctrinal adjunction \cite{Kel74, Luc18x}. The
(codual version of the) techniques of non-canonical isomorphisms from
\cite{Luc19} can be applied just as well to this setting.

\subsection*{Comparison to Cockett and Lack \cite{CL01}}
In \cite{CL01}, the authors address the extensive completion \( \Bool(\cat C) \)
of a distributive category \( \cat C \), whereas our work concerns, among
others, the free lextensive category on any (possibly non-distributive) category
\( \cat C \). 

If \( \cat C \) is already distributive, this raises the question of whether our
completion coincides with Cockett and Lack's construction. The answer is ``no''.
In our setting, the canonical inclusion
\begin{equation}
    \label{eq:free.emb}
    \yy \colon \cat C \to \FinFam(\Lfin(\cat C))
\end{equation}
does \textit{not} preserve finite coproducts nor finite limits, as we are
dealing with a \textit{free} completion. In contrast, the embedding constructed
in \cite{CL01}
\begin{equation*}
    I \colon \cat C \to \Bool(\cat C)
\end{equation*}
preserves coproducts and products, so this is not a free completion. In fact, 
if \( \cat C \) is extensive to begin with, we obtain an equivalence \( I \colon
\cat C \eqv \Bool(\cat C) \), but this is far from the case for the
embedding~\eqref{eq:free.emb}.

This distinction between free and non-free completions is encompassed by the
difference between lax idempotent monads and pseudo-idempotent monads, which is
a topic we plan to discuss in future work.

\bibliographystyle{plain}
\bibliography{LPV24v2}

\end{document}